\documentclass[12pt]{article}%
\usepackage{amsmath, amsthm}
\usepackage{amscd}
\usepackage{graphicx}
\usepackage{amssymb}
\usepackage{amsmath}
\usepackage{amsfonts}%
\providecommand{\U}[1]{\protect\rule{.1in}{.1in}}
\providecommand{\U}[1]{\protect\rule{.1in}{.1in}}
\newtheoremstyle{theorem}
{10pt}
{10pt}
{\sl}
{\parindent}
{\bf}
{. }
{ }
{}
\providecommand{\U}[1]{\protect\rule{.1in}{.1in}}
\providecommand{\U}[1]{\protect\rule{.1in}{.1in}}
\newtheorem{theorem}{Theorem}

\newtheorem{corollary}[theorem]{Corollary}

\newtheorem{definition}[theorem]{Definition}
\newtheorem{example}[theorem]{Example}

\newtheorem{lemma}[theorem]{Lemma}

\newtheorem{proposition}[theorem]{Proposition}
\newtheorem{remark}[theorem]{Remark}

\begin{document}

\title{Shannon-Like Parseval Frame Wavelets on Some Two Step Nilpotent Lie Groups }
\author{Vignon Oussa\\Dept.\ of Mathematics \\Bridgewater State University\\Bridgewater, MA 02325 U.S.A.\\vignon.oussa@bridgew.edu}
\date{}
\maketitle

\begin{abstract}
We construct Shannon-like Parseval frame wavelets on a class of non
commutative two-step nilpotent Lie groups. Our work was inspired by a
construction given by Azita Mayeli on the Heisenberg group. The tools used
here are representation theoretic. However, a great deal of Gabor theory is
used for the construction of the wavelets. The construction obtained here is
very explicit, and we are even able to compute an upper bound for the $L^{2}$
norm for these Parseval frame wavelets.

\end{abstract}



\section{Introduction}

A wavelet on $\mathbb{R}^{n}$ is a function generating an orthonormal basis by
integral shifts and dilations in $L^{2}\left(  \mathbb{R}^{n}\right)  $. At
this point, a great deal is already known about wavelets on $\mathbb{R}^{n}.$
See \cite{Guido} for example. However, the theory of wavelets on non
commutative domains is not as mature. In fact, it is significantly more
difficult to construct wavelets over non commutative groups. Several
mathematicians have made significant contributions to the field. For example,
when it comes to continuous wavelets on locally compact groups, the monograph
by Hartmut F\"{u}hr \cite{Fuhr} is a great source of reference. In
\cite{Currey}, Currey studies continuous wavelets on nilpotent Lie groups, and
Mayeli and Currey developed concepts of wavelet sets on the Heisenberg group
in \cite{CurreyAzita}. Also, in \cite{Lemarie}, wavelets on stratified Lie
groups are studied by Lemari\'{e}. Other imporant contributions can be found
in \cite{Liu}, \cite{Lemarie1}, and \cite{Azita}.

Since the closest objects to $\mathbb{R}^{n}$ are simply connected, connected
non commutative nilpotent Lie groups, it is natural to extend the classical
results of wavelet theory to this class of groups. Even though topologically,
commutative Lie groups, and non commutative simply connected, connected
nilpotent Lie groups are identical, their group structures are quite
different. Let $N$ be a simply connected, and connected nilpotent Lie group.
In order to generate a wavelet system in $L^{2}\left(  N\right)  $, just like
it is done in the classical case, we hope to be able to use a set of
translations and dilations operators acting on either a single function or a
countable set of functions. Naturally, the translation actions should come
from the restriction of the left regular representation of the group to some
discrete set (or lattice) $\Gamma\subset N$. The dilation actions should be
implemented from a discrete subgroup of the outer-automorphism group of $N$.
We would like to remark that the existence of a lattice subgroup in $N$ is in
fact equivalent to the existence of a rational structure on its Lie algebra
\cite{Corwin}.

In this paper, we prove the existence and we give an explicit construction of
some Shannon-like Parseval frame wavelets on a class of step-two nilpotent Lie
groups which we describe as follows. Let $N$ be an $n$-dimensional non
commutative, simply connected, connected, nilpotent Lie group with rational
structure. We assume that the Lie algebra has a fixed Jordan
H\"{o}lder-basis,
\[
J=\left\{  B_{1},\cdots,B_{n}\right\}  =\left\{
\begin{array}
[c]{c}%
Z_{1},Z_{2},\cdots,Z_{n-2d},X_{1},X_{2},\cdots,X_{d},\\
Y_{1},Y_{2},\cdots,Y_{d}%
\end{array}
\right\}  .
\]
The Lie algebra of $N$ has the following properties: $\mathfrak{n=z}\left(
\mathfrak{n}\right)  \oplus\mathfrak{a}\oplus\mathfrak{b,}$ where
$\mathfrak{z}\left(  \mathfrak{n}\right)  =\mathfrak{\mathbb{R}}$-span
$\left\{  Z_{1},Z_{2},\cdots,Z_{n-2d}\right\}  $ is the center of
$\mathfrak{n,}$
\[
\mathfrak{b=\mathbb{R}}-\mathrm{span}\left\{  X_{1},X_{2},\cdots
,X_{d}\right\}  ,
\]
$\mathfrak{a}$ $\mathfrak{=}$ $\mathfrak{\mathbb{R}}$-span $\left\{
Y_{1},Y_{2},\cdots,Y_{d}\right\}  ,$ and $\left[  \mathfrak{a,b}\right]
\subseteq\mathfrak{z}\left(  \mathfrak{n}\right)  .$ Thus, $N$ is isomorphic
to a semi-direct product group of the type $\left(  \mathbb{R}^{n-2d}%
\times\mathbb{R}^{d}\right)  \rtimes\mathbb{R}^{d}.$ Moreover $\exp\left(
\mathfrak{z}\left(  \mathfrak{n}\right)  \oplus\mathfrak{a}\right)  $ is a
maximal commutative normal subgroup of $N,$ and of course, $\mathfrak{z}%
\left(  \mathfrak{n}\right)  \oplus\mathfrak{a}$ is a maximal commutative
ideal of $\mathfrak{n.}$ We also assume that the generic rank of the matrix
\[
\left(
\begin{array}
[c]{ccc}%
\left[  B_{1},B_{1}\right]  & \cdots & \left[  B_{1},B_{n}\right] \\
\vdots & \ddots & \vdots\\
\left[  B_{n},B_{1}\right]  & \cdots & \left[  B_{n},B_{n}\right]
\end{array}
\right)
\]
is equal to $2d$ on $\left[  \mathfrak{n},\mathfrak{n}\right]  .$This class of
groups which contains the Heisenberg groups, direct products of Heisenberg
groups with $\mathbb{R}^{m}$, and various generalizations of the Heisenberg
groups, has also been studied in \cite{Vignon}.

Let $L$ be the left regular representation\textbf{\ }of $N$ acting in
$L^{2}\left(  N\right)  $ endowed with its canonical left Haar measure. We
show that there exists a group
\[
H=\left\{  A^{j}:A=\exp\left(  U\right)  ,j\in\mathbb{Z}\right\}
<\mathrm{Aut}\left(  N\right)  ,
\]
such that $\left[  U,Z_{i}\right]  =\ln\left(  2\right)  Z_{i},$ for $1\leq
i\leq n-2d,$ $\left[  U,Y_{k}\right]  =\ln\left(  2\right)  Y_{k},$ for $1\leq
k\leq d,$ and $\left[  U,X_{k}\right]  =0,$ for $1\leq k\leq d.$ In other
words, $H$ is isomorphic to a discrete subgroup of the automorphism group of
$N.$ Defining a unitary representation of $H$ acting by non-expansive
dilations on $L^{2}\left(  N\right)  ,$ such that $D_{A}:L^{2}\left(
N\right)  \rightarrow L^{2}\left(  N\right)  ,$
\[
f\left(  \cdot\right)  \mapsto\det\left(  Ad_{A}\right)  ^{-1/2}f\left(
A^{-1}\cdot\right)  ,
\]
our main results are summarized in the following terms. We show the existence
of a lattice subgroup which is generated by the discrete set
\[
\Gamma={\prod\limits_{k=1}^{n-2d}}\exp\left(  \mathbb{Z}Z_{k}\right)
\prod\limits_{k=1}^{d}\exp\left(  \mathbb{Z}Y_{k}\right)  \prod\limits_{k=1}%
^{d}\exp\left(  \mathbb{Z}X_{k}\right)  \subset N,
\]
and an infinite countable family of functions $\left\{  f_{k}:k\in
\mathbb{I}\right\}  \subset L^{2}\left(  N\right)  $ such that the system
\[
\left\{  D_{A^{j}}L\left(  \gamma\right)  f_{k}:\gamma\in\Gamma,j\in
\mathbb{Z},k\in\mathbb{I}\right\}
\]
is a Parseval frame in $L^{2}\left(  N\right)  $. For a class of
finite-multiplicity subspaces of $L^{2}\left(  N\right)  $ there exists a
finite number of functions
\[
\left\{  f_{k}:k\in\mathbb{A}\right\}  \subset L^{2}\left(  N\right)
\]
such that the system
\[
\left\{  D_{A^{j}}L\left(  \gamma\right)  f_{k}:\gamma\in\Gamma,j\in
\mathbb{Z},k\in\mathbb{A}\right\}
\]
forms a Parseval frame in $L^{2}\left(  N\right)  $. In particular for
multiplicity-free subspaces, the set $\mathbb{A}$ is a singleton, and in all
cases, $\left\Vert f_{k}\right\Vert _{L^{2}\left(  N\right)  }\leq
2^{\frac{2d-n}{2}}.$

We organize the paper as follows. In the second and third section, we review
some important notions of Gabor theory, and analysis on nilpotent Lie groups.
The main result is proved in the fourth section, and a construction of
Parseval frame wavelets is also given in the same section.

\section{Preliminaries}

A \textbf{lattice} $\Lambda$ in $\mathbb{R}^{2d}$ is a discrete subgroup of
the additive group $\mathbb{R}^{2d}$ that is $\Lambda=M\mathbb{Z}^{2d}$ with
$M$ being a non singular matrix. A \textbf{separable lattice} has the form
$\Lambda=A\mathbb{Z}^{d}\times B\mathbb{Z}^{d}.$ The \textbf{volume} of a
lattice $\Lambda=M\mathbb{Z}^{d}$ equals the Lebesgue measure of
$\mathbb{R}^{d}/\Lambda$ that is $vol\left(  \mathbb{R}^{d}/\Lambda\right)
=\left\vert \det M\right\vert $ and the \textbf{density} of $\Lambda$ is
$d\left(  \Lambda\right)  =\left(  vol\left(  \Lambda\right)  \right)  ^{-1}.$
Let $A\mathbb{Z}^{d}\times B\mathbb{Z}^{d}$ be a separable lattice in
$\mathbb{R}^{2d},$ and $g\in L^{2}\left(  \mathbb{R}^{d}\right)  .$ A
\textbf{Gabor system} $\mathcal{G}\left(  g,A\mathbb{Z}^{d}\times
B\mathbb{Z}^{d}\right)  $ is a sequence of functions defined as
\[
\left\{  e^{2\pi i\left\langle t,y\right\rangle }g\left(  t-x\right)  :y\in
B\mathbb{Z}^{d},\text{ }x\in A\mathbb{Z}^{d}\right\}  .
\]
For example, $\mathcal{G}\left(  \chi_{\left[  0,1\right)  },\mathbb{Z}%
^{d}\times\mathbb{Z}^{d}\right)  $ is a Gabor system, and an orthonormal basis
in $L^{2}\left(  \mathbb{R}^{d}\right)  .$ A sequence $\left\{  f_{n}%
:n\in\mathbb{Z}\right\}  $ of elements in a Hilbert space $H$ is called a
\textbf{frame} if there are constant $A,B>0$ such that
\[
A\left\Vert f\right\Vert ^{2}\leq{\displaystyle\sum\limits_{n\in\mathbb{Z}}%
}\left\vert \left\langle f,f_{n}\right\rangle \right\vert ^{2}\leq B\left\Vert
f\right\Vert ^{2}\text{ for all }f\in H.
\]
The numbers $A,B$ in the definition of a frame are called \textbf{lower} and
\textbf{upper bounds} respectively. A frame is a \textbf{tight frame} if $A=B
$ and a \textbf{normalized tight frame} or \textbf{Parseval frame} if $A=B=1.
$

The following results are well-known in Gabor theory, and can be found in
\cite{Han}.

\begin{proposition}
\label{d}Let $\Lambda$ be a full rank lattice in $\mathbb{R}^{2d}$ with
$d\left(  \Lambda\right)  \geq1$. The following are equivalent

\begin{enumerate}
\item There exists $g\in L^{2}\left(  \mathbb{R}^{d}\right)  $ such that
$\mathcal{G}\left(  g,\Lambda\right)  $ is complete in $L^{2}\left(
\mathbb{R}^{d}\right)  .$

\item There exists $g\in L^{2}\left(  \mathbb{R}^{d}\right)  $ such that
$\mathcal{G}\left(  g,\Lambda\right)  $ is a frame in $L^{2}\left(
\mathbb{R}^{d}\right)  .$

\item There exists $g\in L^{2}\left(  \mathbb{R}^{d}\right)  $ such that
$\mathcal{G}\left(  g,\Lambda\right)  $ is a Parseval frame in $L^{2}\left(
\mathbb{R}^{d}\right)  .$
\end{enumerate}
\end{proposition}

\begin{proposition}
\label{norm} Let $\mathcal{G}\left(  g,A\mathbb{Z}^{d}\times B\mathbb{Z}
^{d}\right)  $ be a Gabor system. If
\[
\mathcal{G}\left(  g,A\mathbb{Z} ^{d}\times B\mathbb{Z}^{d}\right)
\]
is a Parseval frame in $L^{2}\left(  \mathbb{R} ^{d}\right)  $ then
$\left\Vert g\right\Vert ^{2}=\left\vert \det A\det B\right\vert .$
\end{proposition}

\begin{proof}
See proof of Theorem 1.3 in \cite{Han}.
\end{proof}

Let $\psi\in L^{2}\left(  \mathbb{R}\right)  .$ In the classical sense, we say
$\psi$ is a \textbf{wavelet} iff the system
\begin{equation}
\left\{  \psi_{jk}\left(  x\right)  =2^{j/2}\psi\left(  2^{j}x-k\right)
:j,k\in%
\mathbb{Z}
\right\}  \label{wav}%
\end{equation}
forms an orthonormal basis in $L^{2}\left(
\mathbb{R}
\right)  ,$ and we say $\psi$ is a (Parseval) \textbf{frame wavelet }iff the
system (\ref{wav}) forms a (Parseval) frame in $L^{2}\left(
\mathbb{R}
\right)  .$ There are several ways to construct wavelets in $L^{2}\left(
\mathbb{R}
\right)  $. For example the concept of wavelet sets is exposed in
\cite{Speegle}. Also, the oldest known wavelet is the \textbf{Haar wavelet},
given by
\[
\psi\left(  x\right)  =\left\{
\begin{array}
[c]{c}%
1\text{ if }x\in\left[  0,1/2\right) \\
-1\text{ if }x\in\left[  1/2,1\right) \\
0\text{ if }x\in%
\mathbb{R}
\backslash\left(  0,1\right)
\end{array}
\right.  .
\]
The Haar wavelet has been discovered in 1910 way before the concepts of
wavelets were developed. Another well-known example is the function $\psi$
whose Fourier transform is the characteristic function of the
\textbf{Littlewood-Paley wavelet set} $\left[  -2\pi,-\pi\right)  \cup\left[
\pi,2\pi\right)  ,$ and
\[
\psi\left(  x\right)  =\frac{\sin\left(  2\pi x\right)  -\sin\left(  \pi
x\right)  }{\pi x}%
\]
is a wavelet in $L^{2}\left(  \mathbb{R}\right)  .$ The Littlewood-Paley
wavelet set will be of special interest in this paper.

\section{Analysis on nilpotent Lie groups}

The unitary dual of a simply connected, connected nilpotent Lie group is
well-understood via the \textbf{orbit method \cite{Corwin}}. The orbit method
is simply stated in the following terms. Up to isomorphism, the unitary
irreducible representations of any simply connected, connected nilpotent Lie
group are in a one-to-one correspondence with the coadjoint orbits of the Lie
group on elements of the dual of its Lie algebra. In other words, if two
linear functionals belong to the same orbit, their corresponding unitary
irreducible representations must be isomorphic. This correspondence is known
as \textbf{Kirillov's map}. If $N$ is a nilpotent Lie group, the Fourier
transform just like on Euclidean spaces is really defined on $L^{1}\left(
N\right)  \cap L^{2}\left(  N\right)  $ as
\[
\mathcal{F}f\left(  \lambda\right)  =\int_{N}f\left(  n\right)  \pi_{\lambda
}\left(  n\right)  dn
\]
where $\left(  \pi_{\lambda},\mathcal{H}_{\lambda}\right)  $ is an irreducible
representation corresponding to the linear functional $\lambda$ via Kirillov's
map. Clearly, the Fourier transform, as given above is weakly defined and
should be understood as follows. Let $u,v$ be $2$ arbitrary vectors in the
Hilbert space on which we realized the action of $\pi_{\lambda}.$ We have
\[
\left\langle \mathcal{F}f\left(  \lambda\right)  u,v\right\rangle =\int
_{N}f\left(  n\right)  \left\langle \pi_{\lambda}\left(  n\right)
u,v\right\rangle dn.
\]
Since $L^{1}\left(  N\right)  \cap L^{2}\left(  N\right)  $ is dense in
$L^{2}\left(  N\right)  ,$ the extension of the Fourier transform on
$L^{2}\left(  N\right)  $ is naturally called the \textbf{Plancherel
transform}, which we will denote in this paper by $\mathcal{P}.$ $\mathcal{P}$
induces an isometry on $L^{2}\left(  N\right)  ,$ and if $\Sigma$ is a
parametrizing set for the unitary dual of $N,$ there exists a measure called
the Plancherel measure such that
\[
\mathcal{P}\left(  L^{2}\left(  N\right)  \right)  =\int_{\Sigma}^{\oplus
}\left(  \mathcal{H}_{\lambda}\otimes\mathcal{H}_{\lambda}\right)  d\mu\left(
\lambda\right)  .
\]
Moreover, letting $\left(  L,L^{2}\left(  N\right)  \right)  $ be the left
regular representation of $N,$
\[
\mathcal{P}\circ L\circ\mathcal{P}^{-1}=\int_{\Sigma}^{\oplus}\left(
\pi_{\lambda}\otimes1_{\mathcal{H}_{\lambda}}\right)  d\mu\left(
\lambda\right)  ,
\]
where $1_{\mathcal{H}_{\lambda}}$ is the identity operator defined on
$\mathcal{H}_{\lambda},$ we refer the interested reader to \cite{Corwin} which
is a standard reference book for the representation theory of nilpotent Lie groups.

\begin{definition}
Let $\mathfrak{g}$ be an $n$-dimensional Lie algebra over the reals. We say
that $\mathfrak{g}$ has a rational structure if and only if there exists an $%
\mathbb{R}
$-basis $\left\{  V_{1},V_{2},\cdots,V_{n}\right\}  $ for $\mathfrak{g}$
having rational structure constants, and $\mathfrak{g}_{%
\mathbb{Q}
}=%
\mathbb{Q}
$-span$\left\{  V_{1},V_{2},\cdots,V_{n}\right\}  $ provides a rational
structure such that $\mathfrak{g\cong g}_{%
\mathbb{Q}
}\otimes%
\mathbb{R}
.$
\end{definition}

\begin{definition}
Let $G$ be a nilpotent Lie group. A \textbf{lattice subgroup} is a uniform
subgroup $\Gamma$ of $G$ such that $\Lambda=\log\Gamma$ is an additive
subgroup of $\mathfrak{g.}$
\end{definition}

\begin{definition}
Let $G$ be a second-countable, unimodular locally compact group. An
irreducible representation $\pi$ of $G$ acting in $H_{\pi}$ is said to be
square-integrable if for every $u,v\in H_{\pi}$ the matrix coefficient
function $x\mapsto\left\langle \pi\left(  x\right)  u,v\right\rangle $ is in
$L^{2}\left(  G\right)  .$
\end{definition}

Square-integrable modulo the center nilpotent Lie groups are rather appealing
compared to other types of nilpotent Lie groups because their unitary duals
admit simpler descriptions, and they are essentially identified with Zariski
open sets of Euclidean spaces.

Let $N$ be a simply connected, connected, non commutative nilpotent Lie group
with Lie algebra $\mathfrak{n}$ of dimension $n$ over $%
\mathbb{R}
,$ with a rational structure. We start by fixing an ordered Jordan H\"{o}lder
basis
\[
J=\left\{  B_{1},B_{2},\cdots,B_{n}\right\}  =\left\{  Z_{1},Z_{2}%
,\cdots,Z_{n-2d},Y_{1},\cdots,Y_{d},X_{1},\cdots,X_{d}\right\}  ,
\]
such that

\begin{description}
\item[C1.] $\mathfrak{n=z}\left(  \mathfrak{n}\right)  \oplus\mathfrak{a}%
\oplus\mathfrak{b}$ where, $%
\mathbb{R}
$-span $\left\{  Z_{1},Z_{2},\cdots,Z_{n-2d}\right\}  =\mathfrak{z}\left(
\mathfrak{n}\right)  ,$ $%
\mathbb{R}
$-span $\left\{  Y_{1},Y_{2},\cdots,Y_{d}\right\}  =\mathfrak{a},$ and $%
\mathbb{R}
$-span $\left\{  X_{1},X_{2},\cdots,X_{d}\right\}  =\mathfrak{b}.$

\item[C2.] $\mathfrak{z}\left(  \mathfrak{n}\right)  \oplus\mathfrak{a}$ is a
commutative ideal of $\mathfrak{n,}$ and $\mathfrak{b}$ is a commutative
subalgebra (not an ideal) of $\mathfrak{n}$ such that $\left[  \mathfrak{a,b}%
\right]  \subseteq\mathfrak{z}\left(  \mathfrak{n}\right)  .$

\item[C3.] Defining the matrix $M\left(  J\right)  $ of structure constants
related to $J$ such that $M\left(  J\right)  _{i,j}=\left[  B_{i}%
,B_{j}\right]  , $ and letting $\mathbf{0}_{p\times p}$ be the zero matrix of
order $p,$ we assume that
\[
M\left(  J\right)  =\left(
\begin{array}
[c]{ccc}%
\mathbf{0}_{n-2d\times n-2d} & \mathbf{0}_{n-2d\times d} & \mathbf{0}%
_{n-2d\times d}\\
\mathbf{0}_{d\times n-2d} & \mathbf{0}_{d\times d} & -V\\
\mathbf{0}_{d\times n-2d} & V & \mathbf{0}_{d\times d}%
\end{array}
\right)  ,
\]
and we define the matrix%
\[
V=\left(
\begin{array}
[c]{ccc}%
\left[  X_{1},Y_{1}\right]  & \cdots & \left[  X_{1},Y_{d}\right] \\
\vdots & \ddots & \vdots\\
\left[  X_{d},Y_{1}\right]  & \cdots & \left[  X_{d},Y_{d}\right]
\end{array}
\right)  ,
\]
such that $\det\left(  V\right)  $ is a nonzero polynomial with rational
coefficients defined over $\left[  \mathfrak{n,n}\right]  .$
\end{description}

Let $\mathfrak{n}^{\ast}$ the dual vector space of $\mathfrak{n}.$ The
\textbf{coadjoint action} of $N$ on $\mathfrak{n}^{\ast}$ is denoted
multiplicatively such that for any given $\exp W\in N$ and $\lambda
\in\mathfrak{n}^{\ast}$
\begin{equation}
\exp W\cdot\lambda\left(
{\displaystyle\sum\limits_{k=1}^{n}}
u_{k}U_{k}\right)  =\lambda\left(  e^{ad_{-W}}\left(
{\displaystyle\sum\limits_{k=1}^{n}}
u_{k}U_{k}\right)  \right)  . \label{coadjoint}%
\end{equation}

\begin{lemma}
If a Lie group $N$ satisfies conditions C1, C2, and C3, then $N$ is a step
two, square-integrable modulo the center nilpotent Lie group. Moreover, for
any element $\lambda\in\mathfrak{n}^{\ast}$ we define the matrix
\begin{equation}
B\left(  \lambda\right)  =\left(
\begin{array}
[c]{ccc}%
\lambda\left[  X_{1},Y_{1}\right]  & \cdots & \lambda\left[  X_{1}%
,Y_{d}\right] \\
\vdots & \ddots & \vdots\\
\lambda\left[  X_{d},Y_{1}\right]  & \cdots & \lambda\left[  X_{d}%
,Y_{d}\right]
\end{array}
\right)  \label{B}%
\end{equation}
such that $\det B\left(  \lambda\right)  $ is a nonzero polynomial defined
over $\left[  \mathfrak{n,n}\right]  ^{\ast}.$ The dual of $N$ denoted
$\widehat{N}$ is up to a\ null set parametrized by the manifold
\begin{equation}
\Sigma=\left\{  \lambda\in\mathfrak{n}^{\ast}:\det B\left(  \lambda\right)
\neq0,\lambda\left(  X_{i}\right)  =\lambda\left(  Y_{i}\right)  =0,1\leq
i,j\leq d\right\}  , \label{sigma}%
\end{equation}
which is a Zariski open subset of $\mathfrak{z}\left(  \mathfrak{n}\right)
^{\ast}.$ Also, the Plancherel measure (associated to our fixed
Jordan-H\"{o}lder basis) is given by $d\mu\left(  \lambda\right)  =\left\vert
\det B\left(  \lambda\right)  \right\vert d\lambda$ where $d\lambda$ is the
canonical Lebesgue measure defined on $\Sigma.$
\end{lemma}

\begin{proof}
Clearly, $N$ is a step-two nilpotent Lie group because, it is non-commutative
and $\left[  \mathfrak{n},\mathfrak{n}\right]  \subseteq\mathfrak{z}\left(
\mathfrak{n}\right)  .$ Next, $N$ being a nilpotent Lie group, according to
the orbit method, its unitary dual is in one-to-one correspondence with the
coadjoint orbits of $N$ in $\mathfrak{n}^{\ast}$. An algorithm for the
computation of a smooth cross-section parameterizing (up to a null set) almost
all of the irreducible representations is available in Chapter 3 of
\cite{Corwin}. Furthermore, a formula for the computation of the Plancherel
measure is also available in \cite{Corwin} (Chapter 4). To show that $N$ is
square-integrable modulo the center, according to 5.4.4 Corollary in
\cite{Corwin}, it suffices to show that the null-space of the matrix $\left(
\lambda\left[  B_{i},B_{j}\right]  \right)  _{1\leq i,j\leq n}$ is equal to
the central ideal $\mathfrak{z}\left(  \mathfrak{n}\right)  .$ This is clearly
true since by assumption,
\[
\mathrm{rank}\left(  \lambda\left[  B_{i},B_{j}\right]  \right)  _{1\leq
i,j\leq n}=\mathrm{rank}\left(  M\left(  J\right)  \right)  =2d,
\]
and the first $n-2d$ columns of the matrix $\left(  \lambda\left[  B_{i}%
,B_{j}\right]  \right)  _{1\leq i,j\leq n}$ are all zeros (while the remaining
$2d$ columns are linearly independent).
\end{proof}

From now on, we may just assume that $N$ is a simply connected, connected
nilpotent Lie group endowed with a rational structure satisfying conditions
C1, C2 and C3 as defined previously.

\begin{lemma}
For a fixed linear functional $\lambda\in\Sigma$ (see \ref{sigma})$,$ a
corresponding irreducible representation of $N$ is denoted $\pi_{\lambda}$ and
is realized as acting in $L^{2}\left(
\mathbb{R}
^{d}\right)  $ such that for $\phi\in L^{2}\left(
\mathbb{R}
^{d}\right)  ,$
\begin{align}
\pi_{\lambda}\left(  \exp\left(  z_{1}Z_{1}+\cdots+z_{n-2d}Z_{n-2d}\right)
\right)  \phi\left(  t\right)   &  =\exp\left(  2\pi i\left\langle
\lambda,z\right\rangle \right)  \phi\left(  t\right)  , \label{representation}%
\\
\pi_{\lambda}\left(  \exp\left(  y_{1}Y_{1}+\cdots+y_{d}Y_{d}\right)  \right)
\phi\left(  t\right)   &  =\exp\left(  2\pi i\left\langle t,B\left(
\lambda\right)  y\right\rangle \right)  \phi\left(  t\right)  ,\nonumber\\
\pi_{\lambda}\left(  \exp x_{1}X_{1}+\cdots+x_{d}X_{d}\right)  \phi\left(
t\right)   &  =\phi\left(  t-x\right)  ,\nonumber
\end{align}
where $z=\left(  z_{1},\cdots,z_{n-2d}\right)  \in%
\mathbb{R}
^{n-2d},$ $y=\left(  y_{1},\cdots,y_{d}\right)  \in%
\mathbb{R}
^{d},$ $x=\left(  x_{1},\cdots,x_{d}\right)  \in%
\mathbb{R}
^{d}.$
\end{lemma}

For a proof of the lemma, we invite the reader to refer to Chapter 2 in
\cite{Corwin} for general nilpotent Lie groups, or \cite{Vignon} for the class
of groups considered in this paper.

\begin{remark}
From our definition of Gabor systems, for $\phi\in L^{2}\left(
\mathbb{R}
^{d}\right)  ,$
\[
\pi_{\lambda}\left(  \exp\left(
\mathbb{Z}
Y_{1}+\cdots+%
\mathbb{Z}
Y_{d}\right)  \exp\left(
\mathbb{Z}
X_{1}+\cdots+%
\mathbb{Z}
X_{d}\right)  \right)  \phi
\]
is a Gabor system in $L^{2}\left(
\mathbb{R}
^{d}\right)  $ of the type $G\left(  \phi,%
\mathbb{Z}
^{d}\times B\left(  \lambda\right)
\mathbb{Z}
^{d}\right)  .$
\end{remark}

Let $\exp U\in\mathrm{Aut}\left(  N\right)  $ such that $U$ is a derivation of
the Lie algebra of $N.$ For all $i$ such that $1\leq i\leq n$ and for any real
number $\alpha,$ the following must hold

\begin{enumerate}
\item $\left[  U,\alpha B_{i}\right]  =\alpha B_{i},$

\item $\left[  U,B_{i}+B_{j}\right]  =\left[  U,B_{i}\right]  +\left[
U,B_{j}\right]  ,$

\item $\left[  U,B_{i}\right]  =-\left[  B_{i},U\right]  ,$

\item (\textbf{Jacobi identity}) $\left[  \left[  U,B_{i}\right]
,B_{j}\right]  +\left[  \left[  B_{i},B_{j}\right]  ,U\right]  +\left[
\left[  B_{j},U\right]  ,B_{i}\right]  =0.$
\end{enumerate}

\begin{lemma}
\label{Diag}Let $\exp U\in\mathrm{Aut}\left(  N\right)  $ such that $\left[
U,X_{i}\right]  =aX_{i}$ and $\left[  U,Y_{i}\right]  =bY_{i}$ for some
$a,b\in%
\mathbb{R}
$ and for all $i$ such that $1\leq i\leq d.$ There exists a matrix
representation of the linear adjoint action of $%
\mathbb{R}
U$ in $\mathfrak{gl}\left(  \mathfrak{n}\right)  $ such that
\begin{align}
ad  &  :%
\mathbb{R}
U\rightarrow\mathrm{Diag}\left(  \mathfrak{n}\right)  \subset\mathfrak{gl}%
\left(  \mathfrak{n}\right) \label{adU}\\
U  &  \mapsto\left(
\begin{array}
[c]{ccc}%
\left(  a+b\right)  \mathbf{I}_{n-2d} & \cdots & \mathbf{0}\\
\vdots & b\mathbf{I}_{d} & \vdots\\
\mathbf{0} & \cdots & a\mathbf{I}_{d}%
\end{array}
\right)  ,\nonumber
\end{align}
where $\mathbf{I}_{q}$ represents the identity matrix of order $q$.
\end{lemma}

\begin{proof}
To prove the Lemma, it suffices to check axioms 1,2,3, and 4. Clearly, axioms
1,2, and 3 are satisfied. It remains to prove that the Jacobi identity is
satisfied as well. If $\left[  U,X_{i}\right]  =aX_{i}$ and $\left[
U,Y_{i}\right]  =bY_{i},$ by the Jacobi identity,
$
\left[  \left[  X_{i},Y_{j}\right]  ,U\right]  +\left[  \left[  Y_{j}%
,U\right]  ,X_{i}\right]  +\left[  \left[  U,X_{i}\right]  ,Y_{j}\right]
=0$ $\Rightarrow \left[  U,\left[  X_{i},Y_{j}\right]  \right]  =\left(
a+b\right)  \left[  X_{i},Y_{j}\right]  .$
For central elements, there are two separate cases to consider. First, we
suppose that $Z_{i}\in J\ $ is an element of the commutator ideal $\left[
\mathfrak{n},\mathfrak{n}\right]  \leq\mathfrak{z}\left(  \mathfrak{n}\right)
.$ That is $
Z=\sum_{1\leq i,j\leq d}\alpha_{i,j}\left[  X_{i},Y_{j}\right]  .
$ So, $
\left[  U,Z\right]  =\sum_{1\leq i,j\leq d}\alpha_{i,j}\left[  U,\left[
X_{i},Y_{j}\right]  \right]$  $=\left(  a+b\right)  \sum_{1\leq i,j\leq d}%
\alpha_{i,j}\left[  X_{i},Y_{j}\right]  =\left(  a+b\right)  Z.
$
Now assume that $Z=\sum_{1\leq i,j\leq d}\alpha_{i,j}\left[  X_{i}%
,Y_{j}\right]  +W$ such that $W$ is a central element but $W\notin\left[
\mathfrak{n},\mathfrak{n}\right]  .$ Defining $\left[  U,W\right]  =\left(
a+b\right)  W$ does not violate any of the axioms required for $U$ to induce a
linear adjoint action on the Lie algebra $\mathfrak{n.}$ Thus, there exits $U$
such that $\left[  U,Z\right]  =\left(  a+b\right)  Z.$ This completes the proof.
\end{proof}

\begin{corollary}
\label{linear} Let $\exp U\in\mathrm{Aut}\left(  N\right)  $ such that
$\left[  U,X_{i}\right]  =0$ and $\left[  U,Y_{i}\right]  =\ln\left(
2\right)  Y_{i}$ for all $i$ such that $1\leq i\leq d.$ There exists a matrix
representation of the linear adjoint action of $%
\mathbb{R}
U$ in $\mathfrak{gl}\left(  \mathfrak{n}\right)  $ such that%
\begin{align*}
ad_{U}  &  =\left(
\begin{array}
[c]{ccc}%
\ln\left(  2\right)  \mathbf{I}_{n-2d} & \cdots & \mathbf{0}\\
\vdots & \ln\left(  2\right)  \mathbf{I}_{d} & \vdots\\
\mathbf{0} & \cdots & \mathbf{0}_{d}%
\end{array}
\right)  ,\text{ and }\\
Ad_{\exp U}  &  =\left(
\begin{array}
[c]{ccc}%
2\mathbf{I}_{n-2d} & \cdots & \mathbf{0}\\
\vdots & 2\mathbf{I}_{d} & \vdots\\
\mathbf{0} & \cdots & \mathbf{I}_{d}%
\end{array}
\right)
\end{align*}
where $ad_{U}$ is the derivative of $Ad_{\exp U}.$
\end{corollary}

\begin{proof}
For the existence of
\[
ad_{U}=\left(
\begin{array}
[c]{ccc}%
\ln\left(  2\right)  \mathbf{I}_{n-2d} & \cdots & \mathbf{0}\\
\vdots & \ln\left(  2\right)  \mathbf{I}_{d} & \vdots\\
\mathbf{0} & \cdots & \mathbf{0}_{d}%
\end{array}
\right)
\]
we use Lemma \ref{Diag}. Next, since $Ad_{\exp U}=\exp ad_{U},$ we have%
\[
Ad_{\exp U}=\exp\left(
\begin{array}
[c]{ccc}%
\ln\left(  2\right)  \mathbf{I}_{n-2d} & \cdots & \mathbf{0}\\
\vdots & \ln\left(  2\right)  \mathbf{I}_{d} & \vdots\\
\mathbf{0} & \cdots & \mathbf{0}_{d}%
\end{array}
\right)  =\left(
\begin{array}
[c]{ccc}%
2\mathbf{I}_{n-2d} & \cdots & \mathbf{0}\\
\vdots & 2\mathbf{I}_{d} & \vdots\\
\mathbf{0} & \cdots & \mathbf{I}_{d}%
\end{array}
\right)  .
\]
\end{proof}

\begin{proposition}
\label{hom} Let $\mathbf{I}_{m}$ be the $d\times d$ identity matrix. Let
$\phi:%
\mathbb{R}
^{m}\rightarrow%
\mathbb{R}
$ be a \textbf{homogeneous polynomial} and let $\rho$ be the Lebesgue measure
defined on $%
\mathbb{R}
^{m}$. There exits a measurable set $E\subset%
\mathbb{R}
^{m}$ such that the collection of sets $\left\{  \left(  2\mathbf{I}%
_{m}\right)  ^{j}E:j\in%
\mathbb{Z}
\right\}  $ satisfies the following.

\begin{enumerate}
\item $\rho\left(  \left(  2\mathbf{I}_{m}\right)  ^{j}E\cap\left(
2\mathbf{I}_{m}\right)  ^{j^{\prime}}E\right)  =0$ for any $j,j^{\prime}\in%
\mathbb{Z}
,$ and $j\neq j^{\prime}.$

\item $\rho\left(
\mathbb{R}
^{m}-\cup_{j\in%
\mathbb{Z}
}\left(  2\mathbf{I}_{m}\right)  ^{j}E\right)  =0.$

\item $E\subseteq\phi^{-1}\left(  \left[  -1,1\right]  \right)  .$
\end{enumerate}
\end{proposition}

\begin{proof}
First, notice that (1), and (2) together is equivalent to the fact that the
collection of sets $\left\{  \left(  2\mathbf{I}_{m}\right)  ^{j}E:j\in%
\mathbb{Z}
\right\}  $ forms a measurable partition of $%
\mathbb{R}
^{m}.$ Let $S=\left[  -1/2,1/2\right]  ^{m}\backslash\left[  -1/4,1/4\right]
^{m}\subset\mathbb{R}^{m}.$ Clearly $S$ satisfies conditions (1) and (2). If
$S\subseteq\phi^{-1}\left(  \left[  -1,1\right]  \right)  $ then, we are done.
Now assume that $S\nsubseteq\phi^{-1}\left[  -1,1\right]  .$ $\phi$ being a
continuous map, there exists $\epsilon>0$ such that $\phi^{-1}\left[
-1,1\right]  \supset\left(  -\epsilon,\epsilon\right)  ^{m}.$ We pick
$j=j\left(  \epsilon\right)  \in%
\mathbb{Z}
$ such that $\left(  2\mathbf{I}_{m}\right)  ^{j\left(  \epsilon\right)
}S\subset\left(  -\epsilon,\epsilon\right)  ^{m},$ and we let $E=\left(
2\mathbf{I}_{m}\right)  ^{j\left(  \epsilon\right)  }S.$ It is now clear that
$E$ satisfies conditions (1),(2), and (3), and the proposition is proved.
\end{proof}

\begin{remark}
From now on, we identify the spectrum, $\Sigma$ (\ref{sigma}) with an open
subset of $%
\mathbb{R}
^{n-2d}\cong\mathfrak{z}\left(  \mathfrak{n}\right)  ^{\ast}.$ Thus, we abuse
the notation when we make the following statement:
\begin{equation}
\Sigma=\left\{  \left(  \lambda_{1},\cdots,\lambda_{n-2d},0,\cdots,0\right)
:\lambda_{i}\in%
\mathbb{R}
\right\}  \subset%
\mathbb{R}
^{n-2d} \label{identification}%
\end{equation}
It should be understood that, we are assuming that the correct identification
is made for $\Sigma$ by suppressing all of the zero coordinates$.$ Otherwise,
(\ref{identification}) makes no sense of course.
\end{remark}

Now, we would like to specialize Prop \ref{hom} to the class of nilpotent Lie
groups considered in this paper.

\begin{lemma}
\label{density} Identifying $\Sigma$ with a Zariski open subset of $%
\mathbb{R}
^{n-2d}\equiv\mathfrak{z}\left(  \mathfrak{n}\right)  ^{\ast},$ there exists a
Jordan-Holder basis for the Lie algebra $\mathfrak{n}$ of $N$ such that for
\[
\lambda\in\left(  \left[  -1/2,1/2\right]  ^{n-2d}\backslash\left[
-1/4,1/4\right]  ^{n-2d}\right)  \cap\Sigma,
\]
we have $\left\vert \det B\left(  \lambda\right)  \right\vert \leq1.$
\end{lemma}

\begin{proof}
Referring to (\ref{B}) we recall that
\[
B\left(  \lambda\right)  =\left(
\begin{array}
[c]{ccc}%
\lambda\left[  X_{1},Y_{1}\right]  & \cdots & \lambda\left[  X_{1}%
,Y_{d}\right] \\
\vdots & \ddots & \vdots\\
\lambda\left[  X_{d},Y_{1}\right]  & \cdots & \lambda\left[  X_{d}%
,Y_{d}\right]
\end{array}
\right)  .
\]
Let $\mathbf{r}$ be a homogeneous polynomial over $\Sigma$ such that
$\mathbf{r}$ $\mathbf{:}$ $\Sigma\rightarrow%
\mathbb{R}
,$ $\lambda\mapsto\det B\left(  \lambda\right)  .$ From Lemma \ref{hom}, there
exits a measurable set $E\subset\Sigma$ (identified with a conull subset of $%
\mathbb{R}
^{n-2d}$) such that the collection of sets $\left\{  2\mathbf{I}_{m}%
^{j}\left(  E\right)  :j\in%
\mathbb{Z}
\right\}  $ satisfies all of the conditions stated in Lemma \ref{hom}. If
\begin{equation}
\left(  \left[  -1/2,1/2\right]  ^{n-2d}\backslash\left[  -1/4,1/4\right]
^{n-2d}\right)  \cap\Sigma\subset\mathbf{r}^{-1}\left[  -1,1\right]  ,
\label{r}%
\end{equation}
we are done. Otherwise, it not too hard to see that we can modify at least one
element of the Jordan-H\"{o}lder basis
\[
J=\left\{  Z_{1},Z_{2},\cdots,Z_{n-2d},Y_{1},\cdots,Y_{d},X_{1},\cdots
,X_{d}\right\}
\]
to satisfy equation \ref{r}$.$ In fact, let us assume that
\[
\left(  \left[  -1/2,1/2\right]  ^{n-2d}\backslash\left[  -1/4,1/4\right]
^{n-2d}\right)  \cap\Sigma
\]
is not contained in $\mathbf{r}^{-1}\left[  -1,1\right]  .$ Clearly,
\[
\left[  -1/2,1/2\right]  ^{n-2d}\cap\Sigma\cap\mathbf{r}^{-1}\left[
-1,1\right]
\]
$\text{ is not empty.}$ Let $\frac{1}{k}\in%
\mathbb{Q}
,$ and replace $X_{1}$ with $\frac{1}{k}X_{1},$ such that $\mathbf{r}%
_{k}\left(  \lambda\right)  =\det B_{k}\left(  \lambda\right)  $ where
\[
\det B_{k}\left(  \lambda\right)  =\det\left(
\begin{array}
[c]{ccc}%
\frac{1}{k}\lambda\left(  \left[  X_{1},Y_{1}\right]  \right)  & \cdots &
\frac{1}{k}\lambda\left(  \left[  X_{1},Y_{d}\right]  \right) \\
\vdots & \ddots & \vdots\\
\lambda\left(  \left[  X_{d},Y_{1}\right]  \right)  & \cdots & \lambda\left(
\left[  X_{d},Y_{d}\right]  \right)
\end{array}
\right)  =\frac{1}{k}\det B\left(  \lambda\right)  .
\]
As $k\rightarrow\infty,$ $\mathbf{r}_{k}^{-1}\left[  -1,1\right]
\rightarrow\Sigma.$ Certainly, since $$\left[  -1/2,1/2\right]  ^{n-2d}%
\backslash\left[  -1/4,1/4\right]  ^{n-2d}$$ is bounded, for $k$ large enough,
we obtain
\[
\left(  \left[  -1/2,1/2\right]  ^{n-2d}\backslash\left[  -1/4,1/4\right]
^{n-2d}\right)  \cap\Sigma\subseteq\mathbf{r}_{k}^{-1}\left(  \left[
-1,1\right]  \right)  .
\]
Finally, replacing
\[
J=\left\{  Z_{1},Z_{2},\cdots,Z_{n-2d},Y_{1},\cdots,Y_{d},X_{1},\cdots
,X_{d}\right\}
\]
with $\left\{  Z_{1},Z_{2},\cdots,Z_{n-2d},Y_{1},\cdots,Y_{d},\frac{1}{k}%
X_{1},\cdots,X_{d}\right\}  ,$ we complete the proof.
\end{proof}

Now, we make a choice of coordinates which will be convenient for our purpose
throughout this paper. Let $m\in N$ and $\exp U\in\mathrm{Aut}\left(
N\right)  $. We write the automorphic action induced by $A$ as follows
$Ad_{A}\left(  m\right)  =A\left(  m\right)  .$ Given
\[
m=\exp\left(  \sum_{l=1}^{n-2d}z_{l}Z_{l}\right)  \exp\left(  \sum_{k=1}%
^{d}y_{k}Y_{k}\right)  \exp\left(  \sum_{k=1}^{d}x_{k}X_{k}\right)
\]
which we identify with the vector $\left(  z_{1},\cdots,z_{n-2d},y_{1}%
,\cdots,y_{d},x_{1},\cdots,x_{d}\right)  ,$ it is easy to see that%
\[
A\left(  m\right)  =\exp\left(  \sum_{l=1}^{n-2d}2z_{l}Z_{l}\right)
\exp\left(  \sum_{k=1}^{d}2y_{k}Y_{k}\right)  \exp\left(  \sum_{k=1}^{d}%
x_{k}X_{k}\right)  .
\]

\begin{remark}
Identifying $\Sigma$ with a subset of $%
\mathbb{R}
^{n-2d},$ let $A=\exp U$ , and $\lambda\in\Sigma.$ The coadjoint action of $A
$ on $\lambda$ is computed as follows
\[
A\cdot\lambda=\exp U\cdot\left(  \lambda_{1},\cdots,\lambda_{n-2d}\right)
=\frac{1}{2}\mathbf{I}_{n-2d}\left(  \lambda_{1},\cdots,\lambda_{n-2d}\right)
.
\]

\end{remark}

\begin{definition}
We define
\begin{align}
\Gamma_{1} &  =\exp\left(  \sum_{l=1}^{n-2d}%
\mathbb{Z}
Z_{l}\right)  \in\mathfrak{z}\left(  \mathfrak{n}\right)  ,\Gamma_{2}%
=\exp\left(  \sum_{i=1}^{d}%
\mathbb{Z}
Y_{i}\right)  ,\text{ }\nonumber\\
\text{and }\Gamma_{3} &  =\exp\left(  \sum_{k=1}^{d}%
\mathbb{Z}
X_{k}\right)  .
\end{align}
such that $\Gamma=\Gamma_{1}\Gamma_{2}\Gamma_{3}$ and
\begin{equation}
\Gamma=\exp\left(  \sum_{l=1}^{n-2d}%
\mathbb{Z}
Z_{l}\right)  \exp\left(  \sum_{i=1}^{d}%
\mathbb{Z}
Y_{i}\right)  \exp\left(  \sum_{k=1}^{d}%
\mathbb{Z}
X_{k}\right)  .\label{gamma}%
\end{equation}

\end{definition}

\begin{lemma}
The group generated by the set $\Gamma$ as defined in (\ref{gamma}) is a
lattice subgroup of $N.$
\end{lemma}

The proof is elementary. Thus we will omit it. The interested reader is
referred to Chapter $5$ of the book \cite{Corwin}

\section{Existence and Construction of Wavelets}

In this section, taking advantage of the representation theory of nilpotent
Lie groups, we will provide an explicit construction of wavelets over non
commutative nilpotent domains. Let $\mu$ denote the Plancherel measure for the
group $N$ and $\mathcal{P}$ the Plancherel transform defined on $L^{2}\left(
N\right)  .$ From now on, we set
\begin{equation}
E=\left(  \left[  -1/2,1/2\right]  ^{n-2d}\backslash\left[  -1/4,1/4\right]
^{n-2d}\right)  \cap\Sigma. \label{wset}%
\end{equation}
Also, recall that via the Fourier transform, given $x\in N,$ we have
$\mathcal{F}\left(  L\left(  x\right)  f\right)  \left(  \lambda\right)
=\pi_{\lambda}\left(  x\right)  \mathcal{F}f\left(  \lambda\right)  $ and
\[
\mathcal{P}\left(  L^{2}\left(  N\right)  \right)  =\int_{\Sigma}^{\oplus
}L^{2}\left(
\mathbb{R}
^{d}\right)  \otimes L^{2}\left(
\mathbb{R}
^{d}\right)  \left\vert \det B\left(  \lambda\right)  \right\vert d\lambda.
\]
Let $\left\{  P_{\lambda}:\lambda\in\Sigma\right\}  $ be a field of
projections defined on $L^{2}\left(
\mathbb{R}
^{d}\right)  .$ We say that a left-invariant closed Hilbert subspace of
$L^{2}\left(  N\right)  $ is a \textbf{multiplicity-free} subspace if and only
if under the Plancherel transform, the Hilbert space corresponds to
\[
\int_{\Sigma}^{\oplus}\left(  L^{2}\left(
\mathbb{R}
^{d}\right)  \otimes P_{\lambda}\left(  L^{2}\left(
\mathbb{R}
^{d}\right)  \right)  \right)  \left\vert \det B\left(  \lambda\right)
\right\vert d\lambda,
\]
and for almost every $\lambda\in\Sigma,$ $\mathrm{rank}\left(  P_{\lambda
}\right)  =1.$ Similarly, we say that a left-invariant closed subspace of
$L^{2}\left(  N\right)  $ is of \textbf{finite multiplicity} if and only if
the image of the Hilbert space under the Plancherel transform is equal to
\[
\int_{\Sigma}^{\oplus}\left(  L^{2}\left(
\mathbb{R}
^{d}\right)  \otimes P_{\lambda}\left(  L^{2}\left(
\mathbb{R}
^{d}\right)  \right)  \right)  \left\vert \det B\left(  \lambda\right)
\right\vert d\lambda,
\]
such that for almost every $\lambda\in\Sigma,$ $\mathrm{rank}\left(
P_{\lambda}\right)  $ is finite. In this section, we will deal with the
existence and construction of Parseval wavelet frames for
\textbf{multiplicity-free}, and \textbf{finite multiplicity }closed
left-invariant subspaces of $L^{2}\left(  N\right)  .$

We fix $\mathbf{u\in}$ $L^{2}\left(
\mathbb{R}
^{d}\right)  $ such that $\left\Vert \mathbf{u}\right\Vert _{L^{2}\left(
\mathbb{R}
^{d}\right)  }=1.$ Let
\[
\left\{  \mathbf{u}\left(  \lambda\right)  =\mathbf{u}:\lambda\in E\right\}
\]
be a measurable field of unit vectors in $L^{2}\left(
\mathbb{R}
^{d}\right)  .$

\begin{proposition}
\label{V0} Let $N$ be a simply connected, connected nilpotent Lie group
satisfying C1,C2, and C3. We fix a Jordan-H\"{o}lder basis for $\mathfrak{n}$
such that for a.e. linear functional $\lambda\in$ $E$, we have $\left\vert
\det B\left(  \lambda\right)  \right\vert \leq1.$ There exists a bandlimited
function $f\in\mathbf{H}_{E},$ defined as
\begin{align*}
\mathbf{H}_{E}  &  =\mathcal{P}^{-1}\left(  \int_{\Sigma}^{\oplus}L^{2}\left(
%
\mathbb{R}
^{d}\right)  \otimes\left(  \mathbf{u}\chi_{E}\left(  \lambda\right)  \right)
\left\vert \det B\left(  \lambda\right)  \right\vert d\lambda\right) \\
&  =\mathcal{P}^{-1}\left(  \int_{E}^{\oplus}\left(  L^{2}\left(
\mathbb{R}
^{d}\right)  \otimes\mathbf{u}\right)  \left\vert \det B\left(  \lambda
\right)  \right\vert d\lambda\right)
\end{align*}
such that the system $\left\{  L\left(  \gamma\right)  f:\gamma\in
\Gamma\right\}  $ forms a \textbf{Parseval frame} in $\mathbf{H}_{E}.$
Moreover, $\left\Vert f\right\Vert _{\mathbf{H}_{E}}=\mu\left(  E\right)
^{1/2}\leq2^{\frac{2d-n}{2}}. $
\end{proposition}

\begin{proof}
For every fixed linear functional $\lambda\in E$ (see \ref{wset}) since
$\left\vert \det B\left( \lambda\right)\right\vert$ is less or equal to one, there exists a
function $\phi\left(  \lambda\right)  \in L^{2}\left(
\mathbb{R}
^{d}\right)  $ such that the Gabor system $\mathcal{G}\left(  \phi\left(
\lambda\right)  ,%
\mathbb{Z}
^{2}\times\det B\left(  \lambda\right)
\mathbb{Z}
^{2}\right)  $ forms a Parseval frame a.e. This is due to the density
condition given in Proposition \ref{d}. Now, let us define a function
$f\in\mathbf{H}_{E}$ such that
\[
\mathcal{F}f\left(  \lambda\right)  =\left(  \left\vert \det B\left(
\lambda\right)  \right\vert ^{-1/2}\phi\left(  \lambda\right)  \otimes
\mathbf{u}\right)  \chi_{E}\left(  \lambda\right)  .
\]
Referring to (\ref{representation}), we recall that
$
\mathcal{F}\left(  L\left(  \gamma\right)  f\right)  \left(  \lambda\right)
=\pi_{\lambda}\left(  \gamma\right)  \circ\mathcal{F}f\left(  \lambda\right)$ $
=\left\vert \det B\left(  \lambda\right)  \right\vert ^{-1/2}\left(  \left(
\pi_{\lambda}\left(  \gamma\right)  \phi\left(  \lambda\right)  \right)
\otimes\mathbf{u}\right)  \chi_{E}\left(  \lambda\right)  .
$
Let $g$ be any function in $\mathbf{H}_{E}$ such that $\mathcal{F}g\left(
\lambda\right)  =u_{g}\left(  \lambda\right)  \otimes\left(  \mathbf{u}%
\chi_{E}\left(  \lambda\right)  \right)  ,$ and $r(\lambda)=\left\vert \det
B\left(  \lambda\right)  \right\vert ^{1/2}.$ We have
$$
\sum_{\gamma\in\Gamma}\left\vert \left\langle g,L\left(  \gamma\right)
f\right\rangle _{\mathbf{H}_{E}}\right\vert ^{2} =\sum_{\gamma\in\Gamma
}\left\vert \int_{E}\left\langle u_{g}\left(  \lambda\right)  ,r(\lambda
)\pi_{\lambda}\left(  \gamma\right)  \phi\left(  \lambda\right)  \right\rangle
_{L^{2}\left(  \mathbb{R}^{d}\right)  }d\lambda\right\vert ^{2}$$ which is equal to $$  \sum_{\eta\in\mathbb{Z}^{2d}}\sum_{k\in\mathbb{Z}^{n-2d}}\left\vert
\int_{E}e^{2\pi i\left\langle k,\lambda\right\rangle }\left\langle
u_{g}\left(  \lambda\right)  ,r(\lambda)\pi_{\lambda}\left(  \eta\right)
\phi\left(  \lambda\right)  \right\rangle _{L^{2}\left(  \mathbb{R}%
^{d}\right)  }d\lambda\right\vert ^{2}.$$
Since $\left\{  e^{2\pi i\left\langle k,\lambda\right\rangle }\chi_{E}\left(
\lambda\right)  \right\}  _{k\in%
\mathbb{Z}
}$ defines a Parseval frame in $L^{2}\left(  E\right)  ,$ letting
\[
c_{\eta}\left(  \lambda\right)  =\left\langle u_{g}\left(  \lambda\right)
,\left\vert \det B\left(  \lambda\right)  \right\vert ^{1/2}\pi_{\lambda
}\left(  \eta\right)  \phi\left(  \lambda\right)  \right\rangle _{L^{2}\left(
%
\mathbb{R}
^{d}\right)  },
\]
we obtain
$$
\sum_{\eta\in\Gamma_{2}\Gamma_{3}}\sum_{k\in%
\mathbb{Z}
^{n-2d}}\left\vert \int_{E}e^{2\pi i\left\langle k,\lambda\right\rangle
}c_{\eta}\left(  \lambda\right)  d\lambda\right\vert ^{2}=\sum_{\eta\in
\Gamma_{2}\Gamma_{3}}\sum_{k\in%
\mathbb{Z}
^{n-2d}}\left\vert \widehat{c}_{\eta}\left(  k\right)  \right\vert ^{2}$$
The above equality is simply $\sum_{\eta\in\Gamma_{2}\Gamma_{3}}\left\Vert c_{\eta}\right\Vert
_{L^{2}\left(  E\right)  }^{2}.
$
Using the fact that $\mathcal{G}\left(  \phi\left(  \lambda\right)  ,%
\mathbb{Z}
^{2}\times\det B\left(  \lambda\right)
\mathbb{Z}
^{2}\right)  $ is a Parseval frame for almost every $\lambda\in\Sigma$, we
obtain $\sum_{\gamma\in\Gamma}\left\vert \left\langle g,L\left(  \gamma\right)
f\right\rangle _{\mathbf{H}_{E}}\right\vert ^{2}$ is equal to
\begin{align*}
& \int_{E}\sum_{\eta
\in\Gamma_{2}\Gamma_{3}}\left\vert \left\langle u_{g}\left(  \lambda\right)
,\left\vert \det B\left(  \lambda\right)  \right\vert ^{1/2}\pi_{\lambda
}\left(  \eta\right)  \phi\left(  \lambda\right)  \right\rangle _{L^{2}\left(
%
\mathbb{R}
^{d}\right)  }\right\vert ^{2}d\lambda\\
&  =\int_{E}\sum_{\eta\in\Gamma_{2}\Gamma_{3}}\left\vert \left\langle
u_{g}\left(  \lambda\right)  ,\pi_{\lambda}\left(  \eta\right)  \phi\left(
\lambda\right)  \right\rangle _{L^{2}\left(
\mathbb{R}
^{d}\right)  }\right\vert ^{2}\left\vert \det B\left(  \lambda\right)
\right\vert d\lambda\\
&  =\int_{E}\left\Vert \mathcal{F}g\left(  \lambda\right)  \right\Vert
_{\mathcal{HS}}^{2}\left\vert \det B\left(  \lambda\right)  \right\vert
d\lambda\\
&  =\left\Vert g\right\Vert _{\mathbf{H}_{E}}^{2}.
\end{align*}
Now, computing the norm of $f,$ we apply the results from Proposition
\ref{norm}, and we obtain
\[
\left\Vert f\right\Vert _{\mathbf{H}_{E}}^{2}=\int_{E}\left\Vert \phi\left(
\lambda\right)  \right\Vert _{\mathcal{HS}}^{2}d\lambda=\int_{E}\left\vert
\det B\left(  \lambda\right)  \right\vert d\lambda=\mu\left(  E\right)  .
\]
Finally, for the last part, since $E\subset\left\{  \lambda\in\Sigma
:\left\vert \det B\left(  \lambda\right)  \right\vert \leq1\right\}  ,$ then
\[
\left\Vert f\right\Vert _{\mathbf{H}_{E}}^{2}=\int_{E}\left\vert \det B\left(
\lambda\right)  \right\vert d\lambda\leq\int_{E}d\lambda\leq\frac{1}{2^{n-2d}%
}.
\]
This concludes the proof.
\end{proof}

We would like to remark that the Hilbert space $\mathbf{H}_{E}$ is natually
identified with the Hilbert space
\[
L^{2}\left(  E\times%
\mathbb{R}
^{d},\left\vert \det B\left(  \lambda\right)  \right\vert d\lambda dt\right)
.
\]
However, it is much more convenient to use the notation of direct integral. We
recall that given $\gamma\in\Gamma$,
\[
A\left(  \gamma\right)  =\exp\left(  \sum_{k=1}^{n-2d}2m_{k}Z_{k}\right)
\exp\left(  \sum_{k=1}^{d}2n_{k}Y_{k}\right)  \exp\left(  \sum_{k=1}^{d}%
j_{k}X_{k}\right)  .
\]

\begin{definition}
(\textbf{Dilation action}) Let $\mathcal{U}\left(  L^{2}\left(  N\right)
\right)  $ be the group of unitary operators acting in $L^{2}\left(  N\right)
. $ We define a unitary representation of the group\ $H=\left\{  A^{j}:j\in%
\mathbb{Z}
\right\}  $ acting in $L^{2}\left(  N\right)  $ as follows. $D:H\rightarrow
\mathcal{U}\left(  L^{2}\left(  N\right)  \right)  $ and
\[
\left(  D_{A^{j}}f\right)  \left(  x\right)  =\det\left(  Ad_{A}\right)
^{-j/2}f\left(  A^{-j}\left(  x\right)  \right)  =2^{-j\left(  \frac{n-d}%
{2}\right)  }f\left(  A^{-j}\left(  x\right)  \right)  .
\]

\end{definition}

\begin{lemma}
Given any function $f\in L^{2}\left(  N\right)  ,$
\[
\mathcal{F}\left(  D_{A^{j}}f\right)  \left(  \lambda\right)  =2^{j\left(
\frac{n-d}{2}\right)  }\mathcal{F}f\left(  2^{j}\mathbf{I}_{n-2d}%
\lambda\right)  .
\]

\end{lemma}

\begin{proof}
Given any $u,v\in L^{2}\left(
\mathbb{R}
^{d}\right)  ,$%
\begin{align*}
\left\langle \mathcal{F}\left(  D_{A^{j}}f\right)  \left(  \lambda\right)
u,v\right\rangle  &  =\int_{N}D_{A^{j}}f\left(  x\right)  \left\langle
\mathcal{\pi}_{\lambda}\left(  x\right)  u,v\right\rangle dx\\
&  =2^{j\left(  \frac{n-d}{2}\right)  }\int_{N}f\left(  x\right)  \left\langle
\mathcal{\pi}_{\lambda}\left(  A^{j}\left(  x\right)  \right)
u,v\right\rangle dx.
\end{align*}
For each $\lambda\in\Sigma$, we recall that we identify $\lambda=\left(
\lambda_{1},\cdots,\lambda_{n-2d},0,\cdots,0\right)  $ with $\left(
\lambda_{1},\cdots,\lambda_{n-2d}\right)  .$ Also for every $\gamma\in\Gamma,$
we have
\[
\pi_{\lambda}\left(  A^{j}\left(  \gamma\right)  \right)  =\pi_{A^{-j}%
\cdot\lambda}\left(  \gamma\right)  =\pi_{2^{j}\mathbf{I}_{n-2d}\lambda
}\left(  \gamma\right)
\]
where $A^{-j}\cdot\lambda$ denotes the coadjoint action of $A^{-j}$ on a
linear functional $\lambda.$ Since
\[
\left\langle \mathcal{F}\left(  D_{A^{j}}f\right)  \left(  \lambda\right)
u,v\right\rangle =2^{j\left(  \frac{n-d}{2}\right)  }\int_{N}f\left(
x\right)  \left\langle \pi_{2^{j}\mathbf{I}_{n-2d}\lambda}\left(
\gamma\right)  u,v\right\rangle dx,
\]
is true for all $u,v\in L^{2}\left(
\mathbb{R}
^{d}\right)  ,$ we obtain
\[
\mathcal{F}\left(  D_{A^{j}}f\right)  \left(  \lambda\right)  =2^{j\left(
\frac{n-d}{2}\right)  }\mathcal{F}f\left(  2^{j}\mathbf{I}_{n-2d}%
\lambda\right)  .
\]
\end{proof}

\begin{definition}
Given $j\in%
\mathbb{Z}
,$ we define the Hilbert space
\begin{equation}
\mathbf{H}_{E}^{j}=\mathcal{P}^{-1}\left(  \int_{2^{-j}\mathbf{I}_{n-2d}%
E}^{\oplus}\left(  L^{2}\left(
\mathbb{R}
^{d}\right)  \otimes\mathbf{u}\right)  \text{ }\left\vert \det B\left(
\lambda\right)  \right\vert d\lambda\right)  . \label{HJ}%
\end{equation}

\end{definition}

The choice of dilation induced by the action of $A$ has been carefully chosen
in Corollary \ref{linear} so that the following lemma is indeed possible. We
would like to notice that for general dilations, Lemma \ref{dilated} below is
false. That is the dilation coming from $A$ is a special type of dilation.

\begin{lemma}
\label{dilated} For any $j\in%
\mathbb{Z}
,$
\[
\mathcal{P}\left(  D_{A^{j}}\left(  \mathbf{H}_{E}\right)  \right)
\subset\int_{2^{-j}\mathbf{I}_{n-2d}E}^{\oplus}\left(  L^{2}\left(
\mathbb{R}
^{d}\right)  \otimes\mathbf{u}\right)  \left\vert \det B\left(  \lambda
\right)  \right\vert d\lambda.
\]

\end{lemma}

\begin{proof}
Let $f\in\mathbf{H}_{E}.$ For almost every $\lambda\in E,$ there exists
$\phi\left(  \lambda\right)  \in L^{2}\left(
\mathbb{R}
^{d}\right)  $ such that $\mathcal{F}f\left(  \lambda\right)  =\phi\left(
\lambda\right)  \otimes\left(  \mathbf{u}\left(  \chi_{E}\left(
\lambda\right)  \right)  \right)  .$ Next we have
\[
\mathcal{F}\left(  D_{A^{j}}f\right)  \left(  \lambda\right)  =2^{j\left(
\frac{n-d}{2}\right)  }\phi\left(  2^{j}\mathbf{I}_{n-2d}\lambda\right)
\otimes\left(  \mathbf{u}\left(  \chi_{E}\left(  2^{j}\mathbf{I}_{n-2d}%
\lambda\right)  \right)  \right)  .
\]
\end{proof}

\begin{proposition}
For $j,j^{\prime}\in%
\mathbb{Z}
,$ the following hold.

\begin{enumerate}
\item $L\left(  A^{j}\gamma\right)  D_{A^{j}}=D_{A^{j}}L\left(  \gamma\right)
.$

\item Let $f$ be such that $L\left(  \Gamma\right)  f$ is a Parseval frame in
$\mathbf{H}_{E}.$ The system $\left\{  L\left(  A^{j}\gamma\right)  D_{A^{j}%
}f:\gamma\in\Gamma\right\}  $ forms a Parseval frame for $\mathbf{H}_{E}^{j}.$

\item $\left(  D_{A^{j}}L\left(  \Gamma\right)  \right)  \left(
\mathbf{H}_{E}\right)  =\mathbf{H}_{E}^{j}.$

\item For $j\neq j^{\prime},$ $\mathbf{H}_{E}^{j}$ $\bot$ $\mathbf{H}%
_{E}^{j^{\prime}}.$

\item Let $\Sigma$ be the support of the Plancherel measure of the group $N$
as defined in (\ref{sigma}).
\[%
{\displaystyle\bigoplus\limits_{j\in\mathbb{Z}}}
\mathbf{H}_{E}^{j}=\mathcal{P}^{-1}\left(  \int_{\Sigma}^{\oplus}\left(
L^{2}\left(
\mathbb{R}
^{d}\right)  \otimes\mathbf{u}\right)  \left\vert \det B\left(  \lambda
\right)  \right\vert d\lambda\right)  .
\]

\end{enumerate}
\end{proposition}

\begin{proof}
To prove (1), let $\phi\in L^{2}\left(  N\right)  $, and $\phi_{j}=D_{A^{j}%
}\phi.$ Thus, $$L\left(  A^{j}\left(  \gamma\right)  \right)  D_{A^{j}}%
\phi\left(  x\right)  =D_{A^{j}}L\left(  \gamma\right)  \phi\left(  x\right)
.$$ To prove (2) we first show that the sequence
\[
\left\{  L\left(  A^{j}\left(  \gamma\right)  \right)  D_{A^{j}}f:\gamma
\in\Gamma\right\}
\]
is total in $\mathbf{H}_{E}^{j}.$ Let $g$ be a non zero function in
$\mathbf{H}_{E}^{j}$ such that $g$ is orthogonal to the closure of the span of
$\left\{  L\left(  A^{j}\left(  \gamma\right)  \right)  D_{A^{j}}f:\gamma
\in\Gamma\right\}  .$ In other words, for all $\gamma\in\Gamma,\left\langle
g,L\left(  A^{j}\left(  \gamma\right)  \right)  D_{A^{j}}f\right\rangle =0.$
However,%
\begin{align*}
\left\langle g,L\left(  A^{j}\left(  \gamma\right)  \right)  D_{A^{j}%
}f\right\rangle  &  =\left\langle g,D_{A^{j}}L\left(  \gamma\right)
f\right\rangle \\
&  =\left\langle D_{A^{-j}}g,L\left(  \gamma\right)  f\right\rangle .
\end{align*}
Since $D_{A^{-j}}g\in\mathbf{H}_{E},$ then $D_{A^{-j}}g=0\ $and $g=0.$ That
would be a contradiction. Now, we show that $\left\{  L\left(  A^{j}\left(
\gamma\right)  \right)  D_{A^{j}}f:\gamma\in\Gamma\right\}  $ is a frame. Let
$g$ be an arbitrary element of $\mathbf{H}_{E}^{j}$. We have
\begin{align*}
\sum_{\gamma\in\Gamma}\left\vert \left\langle g,L\left(  A^{j}\left(
\gamma\right)  \right)  D_{A^{j}}f\right\rangle \right\vert ^{2}  &
=\sum_{\gamma\in\Gamma}\left\vert \left\langle g,D_{A^{j}}L\left(
\gamma\right)  f\right\rangle \right\vert ^{2}\\
&  =\left\Vert D_{A^{-j}}g\right\Vert _{\mathbf{H}_{E}}^{2}\\
&  =\left\Vert g\right\Vert _{\mathbf{H}_{E}^{j}}^{2}.
\end{align*}
Part (3) is just a direct consequence of Parts (1), and (2), Part (4) is
obvious by the definition given in (\ref{HJ}). To prove Part (5), we use the
fact that the collection of sets $\left\{  A^{-j}E:j\in%
\mathbb{Z}
\right\}  $ forms a measurable partition of $\Sigma$ and we apply Part
(1),(2),(3) and (4).
\end{proof}

\begin{theorem}
Let $f$ be defined such that $\left\{  L\left(  \gamma\right)  f:\gamma
\in\Gamma\right\}  $ is a Parseval frame for $\mathbf{H}_{E}.$ The system
$\left\{  D_{A^{j}}L\left(  \gamma\right)  f:\gamma\in\Gamma,j\in%
\mathbb{Z}
\right\}  $ is a Parseval frame for the multiplicity-free space
\[%
{\displaystyle\bigoplus\limits_{j\in\mathbb{Z} }}
\mathbf{H}_{E}^{j}=\mathcal{P}^{-1}\left(  \int_{\Sigma}^{\oplus}\left(
L^{2}\left(
\mathbb{R}
^{d}\right)  \otimes\mathbf{u}\right)  \left\vert \det B\left(  \lambda
\right)  \right\vert d\lambda\right)  .
\]

\end{theorem}

Fix an orthonormal basis for $L^{2}\left(
\mathbb{R}
^{d}\right)  ,$ $\left\{  \mathbf{e}_{k}:k\in\mathbb{I}\right\}  \ $where
$\mathbb{I}$ is an infinite countable set. For each $k\in\mathbb{I}$, we
obtain the following measurable field of unit vectors $\left\{  \mathbf{u}%
_{k}\left(  \lambda\right)  =\mathbf{e}_{k}:\lambda\in\Sigma\right\} $ and we
define
\[
\mathbf{H}_{E}^{j,k}=\mathcal{P}^{-1}\left(  \int_{2^{-j}\mathbf{I}_{n-2d}%
E}^{\oplus}\left(  L^{2}\left(
\mathbb{R}
^{d}\right)  \otimes\mathbf{e}_{k}\right)  \left\vert \det B\left(
\lambda\right)  \right\vert d\lambda\right)  .
\]
It is clear that $L^{2}\left(  N\right)  =%
{\displaystyle\bigoplus\limits_{\left(  k,j\right)  \in\mathbb{I\times
}\mathbb{Z}}}
\mathbf{H}_{E}^{j,k}.$

\begin{theorem}
Let $S$ be a finite subset of $\mathbb{I}$. For $k\in S,$ we define $f_{k}$
such that $\left\{  L\left(  \gamma\right)  f_{k}:\gamma\in\Gamma\right\}  $
is a Parseval frame for $\mathbf{H}_{E}^{0,k}.$ The system
\[
\left\{  D_{A^{j}}L\left(  \gamma\right)  f_{k}:\gamma\in\Gamma,j\in%
\mathbb{Z}
,k\in S\right\}
\]
is a Parseval frame for
\[%
{\displaystyle\bigoplus\limits_{\left(  k,j\right)  \in S\mathbb{\times
}\mathbb{Z}}}
\mathbf{H}_{E}^{j,k}.
\]
Furthermore, for $k\in\mathbb{I},$ we define $f_{k}$ such that $\left\{
L\left(  \gamma\right)  f_{k}:\gamma\in\Gamma\right\}  $ is a Parseval frame
for $\mathbf{H}_{E}^{0,k}.$ The system
\[
\left\{  D_{A^{j}}L\left(  \gamma\right)  f_{k}:\gamma\in\Gamma,j\in%
\mathbb{Z}
,k\in\mathbb{I}\right\}
\]
is a Parseval frame for $L^{2}\left(  N\right)  .$
\end{theorem}

\begin{example}
A $9$-dimensional case.
\end{example}

Let $N$ a be Lie group with Lie algebra $\mathfrak{n}$ spanned by the basis
\[
\left\{  Z_{1},Z_{2},Z_{3},Y_{1},Y_{2},Y_{3},X_{1},X_{2},X_{3}\right\}
\]
with the following non-trivial Lie brackets.
\begin{align*}
\left[  X_{1},Y_{1}\right]   &  =Z_{1},\left[  X_{2},Y_{1}\right]
=Z_{2},\left[  X_{3},Y_{1}\right]  =Z_{3},\\
\left[  X_{1},Y_{2}\right]   &  =Z_{2},\left[  X_{2},Y_{2}\right]
=Z_{3},\left[  X_{3},Y_{2}\right]  =Z_{1},\\
\left[  X_{1},Y_{3}\right]   &  =Z_{3},\left[  X_{2},Y_{3}\right]
=Z_{1},\left[  X_{3},Y_{3}\right]  =Z_{2}.
\end{align*}
Let $A=\exp U$ such that for all $1\leq i\leq3,$ $\left[  U,Z_{i}\right]
=\ln\left(  2\right)  Z_{i}$ and $\left[  U,Y_{i}\right]  =\ln\left(
2\right)  Y_{i}.$ The Plancherel measure is
\[
d\mu\left(  \lambda\right)  =\left\vert -\lambda_{1}^{3}+3\lambda_{1}%
\lambda_{2}\lambda_{3}-\lambda_{2}^{3}-\lambda_{3}^{3}\right\vert d\lambda
_{1}d\lambda_{2}d\lambda_{3}%
\]
and is supported on the manifold%
\[
\Sigma=\left\{  \left(  \lambda_{1},\lambda_{2},\lambda_{3},0,\cdots,0\right)
\in\mathfrak{n}^{\ast}:-\lambda_{1}^{3}+3\lambda_{1}\lambda_{2}\lambda
_{3}-\lambda_{2}^{3}-\lambda_{3}^{3}\neq0\right\}
\]
which we identify with a Zariski open subset of $%
\mathbb{R}
^{3}.$ Let
\[
E=\left(  \left[  -1/2,1/2\right)  ^{3}\backslash\left[  -1/4,1/4\right)
^{3}\right)  \cap\Sigma.
\]

For each $\lambda\in E,$ there exists a Gabor system $\mathcal{G}\left(
g\left(  \lambda\right)  ,%
\mathbb{Z}
^{3}\times B\left(  \lambda\right)
\mathbb{Z}
^{3}\right)  $ which is a Parseval frame in $L^{2}\left(
\mathbb{R}
^{3}\right)  $ for
\[
B\left(  \lambda\right)  =\left(
\begin{array}
[c]{ccc}%
\lambda_{1} & \lambda_{2} & \lambda_{3}\\
\lambda_{2} & \lambda_{3} & \lambda_{1}\\
\lambda_{3} & \lambda_{1} & \lambda_{2}%
\end{array}
\right)  .
\]
We define $f\in L^{2}\left(  N\right)  $ such that
\[
\mathcal{F}f\left(  \lambda\right)  =\left(  \left\vert -\lambda_{1}%
^{3}+3\lambda_{1}\lambda_{2}\lambda_{3}-\lambda_{2}^{3}-\lambda_{3}%
^{3}\right\vert ^{-1/2}g\left(  \lambda\right)  \otimes\chi_{\left[
0,1\right]  ^{3}}\right)  \chi_{E}\left(  \lambda\right)  .
\]
The system $\left\{  D_{A^{j}}L\left(  \gamma\right)  f:\gamma\in\Gamma,j\in%
\mathbb{Z}
\right\}  $ forms a Parseval frame in
\[
\mathcal{P}^{-1}\left(  \int_{\Sigma}^{\oplus}L^{2}\left(
\mathbb{R}
^{3}\right)  \otimes\chi_{\left[  0,1\right]  ^{3}}d\mu\left(  \lambda\right)
\right)  ,
\]
and $\left\Vert f\right\Vert _{L^{2}\left(  N\right)  }^{2}=7.76\times10^{-2}.
$ Let
\[
\left\{  \beta_{n,k}\left(  t\right)  =\exp\left(  2\pi i\left\langle
t,n\right\rangle \right)  \chi_{\left[  0,1\right]  ^{3}}\left(  t-k\right)
:n\in%
\mathbb{Z}
^{3},k\in%
\mathbb{Z}
^{3}\right\}
\]
be an orthonormal basis in $L^{2}\left(
\mathbb{R}
^{3}\right)  .$ We define
\[
\mathcal{F}f_{n,k}\left(  \lambda\right)  =\left(  \left\vert -\lambda_{1}%
^{3}+3\lambda_{1}\lambda_{2}\lambda_{3}-\lambda_{2}^{3}-\lambda_{3}%
^{3}\right\vert ^{-1/2}g\left(  \lambda\right)  \otimes\beta_{n,k}\right)
\chi_{E}\left(  \lambda\right)  .
\]
The system
\[
\left\{  D_{A^{j}}L\left(  \gamma\right)  f_{n,k}:\gamma\in\Gamma,\text{ }%
j\in\mathbb{Z},\text{ }\left(  n,k\right)  \in\mathbb{Z}^{3}\times
\mathbb{Z}^{3}\right\}
\]
forms a Parseval frame in $L^{2}\left(  N\right)  .$

\begin{remark}
The tools used for the construction of Parseval frames in this paper are
available to us because this class of nilpotent Lie group admits unitary
irreducible representations which behave quite well. For other type of
nilpotent Lie groups, it is not clear how to generalize our construction. Here
is a fairly simple example where we encounter some major obstructions. Let $N$
be a nilpotent Lie group with Lie algebra $\mathfrak{n}$ over the reals,
spanned by $\left\{  Z,Y,X,W\right\}  $ with the following non-trivial Lie
brackets
\[
\left[  X,Y\right]  =Z,\left[  W,X\right]  =Y.
\]
$N$ is a step-3 nilpotent Lie group, and its dual is parametrized by a Zariski
open subset of $\mathbb{R}^{2}.$ However, notice that the center of this group
is only one-dimensional, and the group itself is not square-integrable mod the
center. The construction provided in this paper does not seem to work even for
such a simple group. The main obstruction is related to the fact that the
irreducible representations corresponding to elements in $N/Z\left(  N\right)
$ do not yield to a Gabor system. There are even step-two nilpotent Lie groups
for which our method does not work. For example, let $N$ be a freely-generated
step-two nilpotent Lie group with $3$ generators $Z_{1},Z_{2},Z_{3},$ with
non-trivial Lie brackets%
\[
\left[  Z_{1},Z_{2}\right]  =Z_{12},\left[  Z_{1},Z_{3}\right]  =Z_{13},\text{
and }\left[  Z_{2},Z_{3}\right]  =Z_{23}.
\]
It remains unclear at this point if it is possible to construct Parseval
frames using similar techniques. $\allowbreak$
\end{remark}

\end{document}